\documentclass[a4paper]{article}

\usepackage{amsmath}
\usepackage{amssymb}
\usepackage{amsthm}

\newtheorem{ntn}{Notation}[section]

\newtheorem{lemma}[ntn]{Lemma}
\newtheorem{theorem}[ntn]{Theorem}

\newtheorem{remark}[ntn]{Remark}

\usepackage{graphicx}
\usepackage{subfigure}
\usepackage[usenames]{xcolor}

\newcommand{\R}{\mathbb{R}}

\newcommand{\eps}{\varepsilon}

\newcommand{\what}{\hat{w}}
\newcommand{\wbar}{\overline{w}}

\usepackage{algpseudocode}
\usepackage{algorithm}
\usepackage{authblk}

\usepackage{multirow}

\allowdisplaybreaks

\begin{document}

\title{Improvable Knapsack Problems\thanks{Partially supported by grants SCHO 1140/3-2 and SCHO 1140/6-3 within the Indo-German DST-DFG Programme.}}

\author[1]{Marc Goerigk\footnote{Corresponding author. Email: m.goerigk@lancaster.ac.uk}}
\author[2]{Yogish Sabharwal}
\author[3]{Anita Sch\"obel}
\author[4]{Sandeep Sen}

\affil[1]{{\small Lancaster University, United Kingdom}}
\affil[2]{{\small IBM Research, Delhi, India}}
\affil[3]{{\small Georg-August-Universit\"at G\"ottingen, Germany}}
\affil[4]{{\small IIT Delhi, India}}

\date{}

\maketitle

\begin{abstract}
We consider a variant of the knapsack problem, where items are available with different possible weights. Using a separate budget for these item improvements, the question is: Which items should be improved to which degree such that the resulting classic knapsack problem yields maximum profit?

We present a detailed analysis for several cases of improvable knapsack problems, presenting constant factor approximation algorithms and two PTAS.
\end{abstract}

{\it Parts of this paper have been published in the extended conference abstract \cite{tamc}.}

\section{Introduction}

We consider an extension of the knapsack problem which allows to use different versions of the same item, where the weight of an item can be reduced. Each such improvement has associated costs, and the total budget that can be spent on improvements in bounded. The problem is to find a choice of improvements, such that the resulting knapsack problem has the maximum possible profit.

The idea of improvable versions of optimization problems is not new in the literature; mostly network problems have been considered so far. To the best of our knowledge, this work is the first to consider the knapsack problem, 
where already the basic problem is NP-hard.

In the following, we briefly summarize the state of research on other improvable problems.

Improvable versions of problems that are originally polynomially solvable are in many cases NP-hard. This motivates the analysis of approximation algorithms. In \cite{krumke2}, several such algorithms are presented for node and edge upgrade strategies for subgraph problems (e.g., minimum spanning tree). Improvable spanning trees have further been studied in \cite{krumke4,krumke5}.

In \cite{krumke6}, improvable network flows are studied. Here, edge capacities may be increased to allow for a better maximum flow in the modified network. They show that for continuous improvements, the problem is polynomially solvable, but becomes NP-hard if an edge can only either be improved or not. Minimum cost flows have been considered in \cite{noltemeier1,noltemeier2}. See also \cite{krumke9,krumke10,zhang2} for more results on network improvement problems.

Further studied improvable problems include location problems \cite{berman1,berman2,zhang4,gassner2}, a multicut problem in directed trees \cite{zhang5}, and bottleneck problems \cite{zhang1,zhang6,zhang7}.

Improvable problems have their roots in trying to better model the decision maker's possible choices. Thus, there are several practical applications of improvable problems in the literature. We refer to railway track upgrading \cite{nachtigall1}, disaster management \cite{grasp} and forest road planning \cite{forest} as examples. 
For the knapsack problem, a project manager might decide to hire temporary staff to reduce the time needed for a task. 
Furthermore, improvable problems appear as subproblems when considering the \textit{query competitiveness} for uncertain optimization problems that allow queries to gain additional information, see \cite{sea}.

{\bf Contributions and outlook.} We provide an overview on the results presented in this paper in Table~\ref{tab:results}. The corresponding problem notation is explained in Section~\ref{sec:general}.
\begin{table}[htbp]
\centering\begin{tabular}{c|c|c}
\multicolumn{2}{c|}{\phantom{improvement}} & \phantom{(linear 3-appr.)}  \\[-3ex]
\multicolumn{2}{c|}{improvement} & result \\[1mm]
\hline
 & & \\[-2ex]
\multirow{4}{*}{single} & continuous & PTAS \\[1mm]
\cline{2-3}
& & \\[-2ex]
 &  & linear 6-appr.  \\
 & discrete & (linear 3-appr.) \\
 & & (poly. 2-appr.)  \\[1mm]
\hline
& &  \\[-2ex]
\multirow{4}{*}{multi} & continuous & open \\[1mm]
\cline{2-3} 
& & \\[-2ex]
 &  & PTAS \\
 & discrete & poly. 4-appr. \\
 & & poly. 3-appr. \\
\end{tabular}
\caption{Results of this paper. Multi-stage results also apply to single-stage problems. Results in brackets ($\cdot$) indicate they hold for special cases.}\label{tab:results}
\end{table}

The remainder of this paper is structured as follows.
We introduce the improvable knapsack problem where several degrees of improvement on item weights are possible in Section~\ref{sec:general}, and discuss notation to differentiate between problem variants. In Sections~\ref{sec:contweights} and \ref{sec:disweight}, we discuss continuous and discrete improvements, respectively. 

We consider the special case of unit improvement costs in Section~\ref{sec:spec}, which admits improved approximation ratios. Section~\ref{sec:conclusion} concludes the paper, discusses extensions of the presented methods to more general improvable knapsack problems, and points to further research questions.

\section{The Improvable Knapsack Problem}
\label{sec:general}

To formalize the improvable knapsack problem, we assume the following setting. Like in the classic knapsack problem, we are given a list of $n$ items $\{1,\ldots,n\}$ with profits $p_i\in\R$ and weights $w_i\in\R$, as well as a budget $B$. For the improvable version, we are furthermore given a list of improved weights $(w_i^1,\ldots,w_i^{j(i)})$ per item $i$ with associated costs $(c_i^1,\ldots,c_i^{j(i)})$ and a budget $C$.
We assume that improved weights are monotonically decreasing, while improvement costs are monotonically increasing. To simplify notation, we set $c^0_i = 0$ for all $i$.

\medskip

The (weight-)improvable knapsack problem (iK) is then given as: For each item, determine the degree of weight improvement, so that the total knapsack profit is maximized, under consideration of the two budgets $B$ and $C$.

\medskip

The problem (iK) can be modeled using the following binary program.
\begin{align}
\max\ & \sum_{i=1}^n p_ix_i \\
\text{s.t. } & \sum_{i=1}^n w_i x_i + \sum_{i=1}^n\sum_{\ell=1}^{j(i)} (w_i^\ell - w_i^{\ell-1}) y_i^\ell \le B \hspace{-5mm}\label{ik:1}\\
& \sum_{i=1}^n\sum_{\ell=1}^{j(i)} (c_i^\ell- c_i^{\ell-1}) y_i^\ell \le C \label{ik:2}\\
& y_i^{j(i)} \le y_i^{j(i)-1}\le \ldots \le y_i^1 \le x_i & \forall i\in\{1,\ldots,n\} \label{ik:4}\\
& x_i \in\{0,1\} & \forall i\in\{1,\ldots,n\} \\
& y_i^\ell \in \{0,1\} & \forall i\in\{1,\ldots,n\}, \ell\in\{1,\ldots,j(i)\}
\end{align}
Whether an item is packed or not is modeled by using the variables $x_i$. 
The variables 
$y^\ell_i$  determine 
the degree of weight improvement. Constraint~\eqref{ik:1} ensures that the knapsack budget is respected, while Constraint~\eqref{ik:2} models the improvement budget for the weights. Finally, Constraints~\eqref{ik:4} 
ensure that a certain degree of improvement can only be used if also the previous degrees of improvement are used.

Note that for $j(i) = 0$ for all $i$, the improvable knapsack problem becomes a classic knapsack problem again. We now introduce some notation and special cases.
\begin{ntn}
\begin{itemize}
\item If $j(i) \le 1 $ for all $i$ (with $j(i)=1$ for at least one $i$), we call $(iK)$ a {\it single-level} problem.
\item If there is at least one $i$ with $j(i) \ge 2$, we call $(iK)$ a {\it multi-level} problem.
\item For a set of indices $I\subseteq\{1,\ldots,n\}$, we denote by $p(I) := \sum_{i\in I} p_i$ its profit.
\item Finally, we say that the above formulation with binary values for $y^\ell_i$ 
is a problem with {\it discrete improvements}. If we relax these variables to take values from $[0,1]$ instead, we say the problem has {\it continuous improvements}.
\end{itemize}
\end{ntn}
Note that one might also consider a similar type of improvement on the profits. This is briefly discussed in Section~\ref{sec:conclusion}.

\section{Continuous Improvements}
\label{sec:contweights}

Adapting formulation (iK), the problem we consider here is as follows:
\begin{align}
\mbox{(iK-cs)} \hspace{1.5cm} \max \ & \sum_{i=1}^n p_i x_i \nonumber \\
\mbox{s.t. } & \sum_{i=1}^n w_i x_i - \sum_{i=1}^n \overline{w}_i y_i  \leq B \label{knap1} \\
& \sum_{i=1}^n c_i y_i \leq C \label{knap2}\\
&y_i \leq x_i & \forall i\in\{1,\ldots,n\} \label{knap3}\\
&y_i \geq 0, \ x_i \in \{0,1\} & \forall i\in\{1,\ldots,n\}
\end{align}
where $\overline{w}_i := w^1_i - w^0_i$ and $c_i := c^1_i$ in the setting of problem (iK). In (iK-cs), the letter 
``c'' stands for continuous, and ``s'' for single-level. We make use of similar notation for other cases throughout the following sections.

Note that (iK-cs) always admits a feasible solution (namely $x_i=0, y_i=0$) hence, an optimal
solution to (iK-cs) exists.

We first show that an optimal $y$ for some given $x$ is easy to compute. We assume without loss of generality that items are sorted with respect to improvement costs per improved weight, i.e., 
\begin{equation}
\label{order}
\frac{c_1}{\overline{w}_1} \leq \cdots \leq \frac{c_n}{\overline{w}_n} 
\end{equation}
 
\begin{lemma}
\label{lem1} 
Let $x\in\{0,1\}^n$ be given, and let $I:=\{i: x_i=1\}$. 
Define
\begin{equation}
\label{y-formel}
 y^*_i:=\left\{ \begin{array}{ll} 
                 1 & \mbox{ if } \sum_{\ell=1}^i c_\ell x_\ell < C \mbox{ and } i \in I\\
   \frac{C-\sum_{\ell=1}^{i-1} c_\ell x_\ell}{c_i} & 
                 \mbox{ if } \sum_{\ell=1}^{i-1} c_\ell x_\ell < C \mbox{ and } \sum_{\ell=1}^i c_\ell x_\ell \geq C \mbox{ and } i \in I\\
                 0 & \mbox{ otherwise}
                \end{array} \right.
\end{equation} 
for all $i\in \{1,\ldots,n\}$. Then we have:
\begin{itemize}
\item If $\sum_{i \in I} \overline{w}_i y^*_i \geq \sum_{i \in I} w_i x_i - B$ then $(x,y^*)$ 
is a feasible solution to (iK-cs).
\item If $\sum_{i \in I} \overline{w}_i y^*_i < \sum_{i \in I} w_i x_i - B$ then there does not exist 
any $y$ such that $(x,y)$ is a feasible to solution (iK-cs).
\end{itemize}
\end{lemma}

\begin{proof}
Let $x$ be fixed. Consider the following continuous knapsack problem (K) with items only in $I$:
\begin{align*}
\mbox{(K)} \hspace{1.5cm} \max \ &\sum_{i \in I} \overline{w}_i y_i \\
\mbox{s.t. }  & \sum_{i \in I} c_i y_i  \leq  C\\
& 0 \leq y_i \leq 1  & \forall i \in I
\end{align*}
As (K) is a continuous knapsack problem on $I$, it 
can be solved by sorting the items of $I$ according to (\ref{order}) and adding items as long 
as the budget $B$ allows. Thus, there is an optimal solution $y^*$ that has the form of (\ref{y-formel}).
In our case, $y^*$ also contains components which refer to items not included in $I$.
These components are set to zero. 

We now prove the assertions of the lemma:
If $\sum_{i \in I} \overline{w}_i y^*_i \geq \sum_{i \in I} w_i x_i - B$ then 
$(x,y^*)$ satisfies all constraints of (iK-cs) and is hence feasible. 
For the second statement, assume that $(x,y)$ is a feasible solution to (iK-cs). Then $y$ is a feasible
solution to (K). Let $y^*$ be an optimal solution $y^*$ to (K). Then we have
\[ \sum_{i \in I} \overline{w}_i y^*_i \geq \sum_{i \in I} \overline{w}_i y_i \geq \sum_{i \in I} w_i x_i - B, \] where the last inequality holds since $(x,y)$ is feasible to (iK-cs) and hence satisfies (\ref{knap1}).
\end{proof}

\begin{lemma}
\label{lem2}
Let $(x,y)$ be a feasible solution to (iK-cs). Then there exists a feasible solution $(x,y^*)$ 
to (iK-cs) and some index $k \in \{1,\ldots,n\}$ such that the following conditions are satisfied:
\begin{enumerate}
\item[a)] $y^*_i=x_i$ for all $i < k$, 
\item[b)] $y^*_i=0$ for all $i > k$,
\item[c)] $x_k=1$,
\item[d)] $0 \ < \ y^*_k \ \leq \ 1$ 
\end{enumerate}
We call $k$ the \emph{fractional index}.
\end{lemma}

\begin{proof}
  Let $(x,y)$ be a feasible solution to (iK-cs). From
  Lemma~\ref{lem1} we know that there exists a solution $y^*$ which is also feasible and
  computed according to (\ref{y-formel}). As before, let  $I=\{i: x_i=1\}$. 
  Choose $k$ as the (unique) index such that 
  $\sum_{i=1}^{k-1} c_i x_i < C$  and $\sum_{i=1}^k c_i x_i \geq C$. 
In case that $\sum_{i=1}^n c_i x_i < C$, choose $k=n$
 (i.e., $k$ is the highest index of an improved item, which is the only one which might be fractional).
  We then know that for all 
  $i \in I$: $y^*_i=1=x_i$ if $i < k$ and $y^*_i=0$ if $i > k$. Furthermore,
  for $i \not \in I$ we have $y^*_i \leq x_i=0$, hence $y^*_i=0$. Therefore, \textit{a)} and \textit{b)}
  are satisfied. 

  From (\ref{y-formel}) we also see that \textit{c)} holds, otherwise
  $\sum_{\ell=1}^{k-1} c_\ell x_\ell=\sum_{\ell=1}^k c_\ell x_\ell$ and 
  $\sum_{\ell=1}^{k-1} c_\ell x_\ell < C$  
  and $\sum_{\ell=1}^k c_\ell x_\ell \geq C$ can never be satisfied simultaneously.
  Finally, \textit{d)} is equivalent to requiring that
  $0 < \frac{C-\sum_{\ell=1}^{k-1} c_\ell x_\ell}{c_k} \leq 1$ which follows 
  from  $\sum_{\ell=1}^{k-1} c_\ell x_\ell < C$  and from $C- \sum_{\ell=1}^{k-1} c_\ell x_\ell \leq c_k$. 
\end{proof}

It is hence enough to look for an optimal solution of (iK-cs) which satisfies the four
conditions of Lemma~\ref{lem2}. 

\begin{lemma}
\label{lem3}
Let $(x,y)$ be a solution satisfying the conditions of Lemma~\ref{lem2}. Let 
$k$ be the fractional index. Then
\begin{align*}
\sum_{i=1}^n w_i x_i - \sum_{i=1}^n \overline{w}_i y_i &\leq  B 
\\
\Longleftrightarrow \ \ \sum_{i=1}^{k-1} (w_i - \overline{w}_i + \overline{w}_k \frac{c_i}{c_k} ) x_i 
+ \sum_{i=k+1}^{n} w_i x_i &\leq B + C \frac{\overline{w}_k}{c_k} - w_k 
\end{align*}
\end{lemma}

\begin{proof}
Using the conditions \textit{a)}, \textit{b)}, and \textit{c)} of Lemma~\ref{lem2} we obtain that
\begin{eqnarray*}
\sum_{i=1}^n w_i x_i - \sum_{i=1}^n \overline{w}_i y_i & = & 
\sum_{i=1}^{n} w_i x_i - \sum_{i=1}^{k-1} \overline{w}_i y_i 
- \overline{w}_k y_k 
- \sum_{i=k+1}^{n} \overline{w}_i y_i \\
& = & 
\sum_{i=1}^{n} w_i x_i 
- \sum_{i=1}^{k-1} \overline{w}_i x_i - \overline{w}_k \frac{C-\sum_{i=1}^{k-1} c_i x_i}{c_k} \\
& = & 
\sum_{i=1}^{k-1} (w_i - \overline{w}_i + \overline{w}_k \frac{c_i}{c_k} ) x_i
+ w_k + \sum_{i=k+1}^{n} w_i x_i 
- C \frac{\overline{w}_k}{c_k}, 
\end{eqnarray*}
hence the result follows.

\end{proof}

Using this result in the formulation of (iK-cs) leaves us with an optimization problem (P($k$)) of the following form:
\begin{align*}
\mbox{(P($k$))} \hspace{0.8cm} \max\ &\sum_{i=1\atop i\neq k }^n p_i x_i \\
\mbox{s.t. } &\sum_{i=1}^{k-1} (w_i - \overline{w}_i ) x_i + \sum_{i=k+1}^{n} w_i x_i \leq B -w_k + \overline{w}_k y\\
& \sum_{i=1}^{k-1} c_i x_i \le C - c_k y_k \\
& x_i \in \{0,1\} & \hspace{-2cm}\forall i\in\{1,\ldots,n\}\setminus\{k\} \\
& y\in [0,1]
\end{align*}
Solving P($k$) for $k=1,\ldots,n$ would give us an optimal solution to (iK-cs). It remains to see how these problems can be treated. 

\begin{lemma}\label{lppk}
Solving the LP relaxation of P($k$) gives a basic solution with at most two fractional variables.
\end{lemma}
\begin{proof}
Note that the LP relaxation of P($k$) in standard form
\begin{align*}
\max\ &\sum_{i=1\atop i\neq k }^n p_i x_i \\
\mbox{s.t. } &\sum_{i=1}^{k-1} (w_i - \overline{w}_i ) x_i + \sum_{i=k+1}^{n} w_i x_i + \alpha = B -w_k + \overline{w}_k y\\
& \sum_{i=1}^{k-1} c_i x_i +\beta =  C - c_k y_k \\
& x_i + \gamma_i = 1 & \forall i\in[n] \\
& y + \delta = 1 \\
& x_i,\gamma_i \ge 0 & \hspace{-2cm}\forall i\in\{1,\ldots,n\}\setminus\{k\} \\
& y,\alpha,\beta,\delta \ge 0
\end{align*}
has $2n+2$ variables, and $n+2$ constraints. Hence, a basis contains $n+2$ variables, leaving $n$ non-basis variables that are equal to zero. If any of the $x_i$, $\gamma_i$, $y$ or $\delta$ variables is a non-basis variable, then the corresponding partner variable is not fractional. Therefore, there can be at most two fractional variables.
\end{proof}

\begin{lemma}\label{ptaslemma}
There is a PTAS for problem P($k$).
\end{lemma}
\begin{proof}
We follow a similar idea as \cite{patt2012vector}.
Let $\eps>0$, and set $q=\min\{n,2/\eps\}$. 
Suppose we could guess the $q$ largest items $Q^+\subseteq Q^*$ that are packed by some optimal solution $Q^*\subseteq[n]$. Let $\underline{p}$ be the smallest profit of these items, i.e., $\underline{p} := \min \{p_i: i\in Q^+\}$.

We construct a heuristic solution $x^*$ in the following way: Set $x^*_i = 1$ for all $i\in Q^+$. We set $Q^- = \{ i\in[n]\setminus Q^+ : p_i > \underline{p}_i\}$ and $x^*_i = 0$ for all $i\in Q^-$. We denote the sub-instance of P($k$) consisting of the remaining items $[n]\setminus(Q^+\cup Q^-)$ as $P(Q^+,Q^-)$. We solve the LP relaxation of $P(Q^+,Q^-)$ and round down all fractional variables of the resulting optimal solution $x^F$. Using this rounded down solution, we fill in the missing values of $x^*$. Note that $x^*$ is feasible for P($k$) by construction.

We now analyze the objective value of such a solution. Solving the linear relaxation of $P(Q^+,Q^-)$ instead of the mixed-binary problem results in an error of at most $2\underline{p}$, as due to Lemma~\ref{lppk} at most two items are rounded down, and every item has profit at most $\underline{p}$.

Let $x'$ denote an optimal solution for P($k$), of which we guessed the $q$ items with highest profit. Then $OPT = \sum_{i\in[n]} p_i x'_i \ge \underline{p}q$ and $x'_i = 0$ for all $i\in Q^-$. We get
\begin{align*}
\sum_{i\in[n]} p_i x^*_i &= \sum_{i\in Q^+} p_i x^*_i + \sum_{i\in Q^-} p_i x^*_i + \sum_{i\in[n]\setminus(Q^+\cup Q^-)} p_i x^*_i\\
&\ge \sum_{i\in Q^+} p_i x^*_i + \sum_{i\in Q^-} p_i x^*_i + \sum_{i\in[n]\setminus(Q^+\cup Q^-)} p_i x^F_i - 2\underline{p} \\
&= \sum_{i\in Q^+} p_i x'_i + \sum_{i\in[n]\setminus(Q^+\cup Q^-)} p_i x^F_i - 2\underline{p} \\
&\ge \sum_{i\in Q^+} p_i x'_i + \sum_{i\in[n]\setminus(Q^+\cup Q^-)} p_i x'_i - 2\underline{p} \\
&= OPT - 2\underline{p} \ge OPT - \frac{OPT}{q} = (1-\varepsilon)OPT
\end{align*}
Now, this is only possible if the correct set $Q^+$ can be guessed. For a constant value of $\varepsilon$, all possible candidate sets can be enumerated in polynomial time. Thus, we have constructed a PTAS for problem P($k$).

\end{proof}

\begin{lemma}\label{combine}
Let $OPT(k)$ be the optimal objective value of P($k$), and let $OPT$ be the optimal objective value of (iK-cs). For every $k$, let a solution $x^k$ be given with profit $\sum_{i\in[n]} p_i x^k_i \ge (1-\varepsilon)OPT(k)$ for some constant $\varepsilon$. Then
\[ \max_k \sum_{i\in[n]} p_i x^k_i \ge (1-\varepsilon)OPT \]
\end{lemma}
\begin{proof}
Follows directly from $OPT = \max_k OPT(k)$.
\end{proof}

\begin{theorem}
There exists a PTAS for problem (iK-cs).
\end{theorem}
\begin{proof}
Let some $\varepsilon>0$ be given. For all $k\in[n]$, we use Lemma~\ref{ptaslemma} to construct a solution that is within $(1-\varepsilon)$ of optimality to P($k$). Using Lemma~\ref{combine}, we find a solution that is within $(1-\varepsilon)$ of optimality for problem (iK-cs) in polynomial time.
\end{proof}

%
%
%
%
%
%

\section{Discrete Weight Improvements}
\label{sec:disweight}

\subsection{Single-Level Case}

We now consider the special case of (iK) where items can be improved at most once, and in a binary fashion. We denote this problem as (iK-ds) and formulate it as a binary program in the following way:
\begin{align}
\text{(iK-ds)} \hspace{1.5cm} \max\ &\sum_{i=1}^n p_i x_i \label{linear2-lp1}\\
\text{s.t. } &\sum_{i=1}^n w_i x_i - \sum_{i=1}^n \wbar_i y_i \le B \label{linear2-lp2}\\
&\sum_{i=1}^n c_i y_i \le C \label{linear2-lp3}\\
&y_i \le x_i & \forall i\in\{1,\ldots,n\} \label{linear2-lp4}\\
&x_i \in \{0,1\} & \forall i\in\{1,\ldots,n\} \\
&y_i \in \{0,1\} & \forall i \in \{1,\ldots,n\} \label{linear2-lp5}
\end{align}
Note that the difference to (iK-cs) is that the improvement variables $y_i$ are now integer.
Furthermore, this special case of a weight-reducible knapsack problem
is related to the multi-dimensional knapsack problem (MKP), with two knapsack constraints:
\[ \text{(MKP)} \quad \max \left\{ \sum_{i=1}^n p_i x_i\ :\ \sum_{i=1}^n w_i x_i \le B,\ \sum_{i=1}^n c_i x_i \le C,\ x\in \{0,1\}^n \right\}\]
which is a well-researched knapsack variant on its own \cite{Freville20041}. The difference lies in the existence of the coupling constraints (\ref{linear2-lp4}). We write $\text{MKP}^*(w,c)$ to denote the optimal objective value of the 2-dimensional knapsack problem with item weights $w$ and $c$.

In the following, we also make use of the following reformulation of (iK-ds):
\begin{align}
\text{(iK-ds')}\hspace{1.5cm} \max\ &\sum_{i=1}^n p_i( x_i + \hat{x}_i) \label{linear1-lp1}\\
\text{s.t. } &\sum_{i=1}^n w_i x_i + \sum_{i=1}^n \hat{w}_i \hat{x}_i \le B \label{linear1-lp2}\\ 
&\sum_{i=1}^n c_i \hat{x}_i \le C \label{linear1-lp3}\\
&x_i + \hat{x}_i \le 1 & \forall i\in\{1,\ldots,n\} \label{linear1-lp4}\\
&x_i ,\hat{x}_i \in\{0,1\} & \forall i\in\{1,\ldots,n\}\label{linear1-lp5}
\end{align}
where $\hat{w}_i:=w^1_i$ denotes the improved item weight, i.e., $x_i$ models an item that is used in its unimproved form, and $\hat{x}_i$ means it is used with improvement. One cannot pack both the unimproved and the improved form.

We now show that there is a linear-time 6-approximation algorithm for this problem.
To this end, we separate (iK-ds') into two subproblems.

In the first problem, we use items only in their improved form. This results in a two-dimensional knapsack problem MKP$(\what,c)$. As the number of constraints is constant, its LP relaxation can be solved in linear time due to \cite{Megiddo93lineartime}. Furthermore, as a basis solution can have at most two fractional solutions, the LP relaxation gives an immediate 3-approximation to the binary problem.

The second subproblem we consider ignores that items can be improved, i.e., we simply solve the knapsack problem with respect to the original item weights $w$. This can be solved using a 2-approximation in linear time.

\begin{theorem}
There is a 6-approximation algorithm for problem (iK-ds) that runs in linear time.
\end{theorem}

\begin{proof}
Let $opt$ denote the optimal objective value of (iK-ds). We show that
\[ opt \le 2\max\{ \text{MKP}^*(\what,c) ,\text{KP}^*(w)\}. \]
From the reasoning above, the theorem then follows.

To this end, let $(x^*,\hat{x}^*)$ denote an optimal solution to (iK-ds'), let $x^1$ be an optimal solution to MKP$(\what,c)$, and let $x^2$ be an optimal solution to KP$(w)$. As $\hat{x}^*$ is feasible for MKP$(\what,c)$, we have that $\sum_{i=1}^n p_i \hat{x}^*_i \le \sum_{i=1}^n p_i x^1_i$. Also, $x^*$ is feasible for KP$(w)$; thus, $\sum_{i=1}^n p_i x^*_i \le \sum_{i=1}^n p_i x^2_i$. Together, we have that
\[ opt = \sum_{i=1}^n p_i (x^*_i + \hat{x}^*_i ) \le \sum_{i=1}^n p_i (x^1_i + x^2_i) \le 2\max\{ \text{MKP}^*(\what,c) ,\text{KP}^*(w)\}. \]
Furthermore, $(0,x^1)$ and $(x^2,0)$ are both feasible for (iK-ds').
\end{proof}

Further algorithms for more specific cases are presented in Section~\ref{sec:spec}. Also, the algorithms presented in the following multi-level case can be applied as well.

\subsection{Multi-level}
\label{sec:dp}

We consider the knapsack problem with multiple, discrete weight improvements (iK-dm), which can be written as
\begin{align}
\text{(iK-dm)} \hspace{.5cm} \max\ & \sum_{i=1}^n p_ix_i \\
\text{s.t. } & \sum_{i=1}^n w_i x_i + \sum_{i=1}^n\sum_{\ell=1}^{j(i)} (w_i^\ell - w_i^{\ell-1}) y_i^\ell \le B \hspace{-1cm}\\
& \sum_{i=1}^n\sum_{\ell=1}^{j(i)} (c_i^\ell- c_i^{\ell-1}) y_i^\ell \le C \\
& y_i^{j(i)} \le y_i^{j(i)-1}\le \ldots \le y_i^1 \le x_i & \forall i\in\{1,\ldots,n\} \label{gcr-a}\\
& x_i \in\{0,1\} & \forall i\in\{1,\ldots,n\} \label{gcr-int1}\\
& y_i^\ell \in \{0,1\} & \hspace{-6mm}\forall i\in\{1,\ldots,n\}, \ell\in\{1,\ldots,j(i)\} \label{gcr-int2}
\end{align}
We begin with an algorithm for integral profits, i.e., for $p_i \in \mathbb{N}$ for all $i\in\{1,\ldots,n\}$ and subsequently use
scaling techniques to obtain a more efficient variation at the expense of a $1+\epsilon$ approximation in the objective value. 
The basic idea of updating a table with relevant problem information can be found in, e.g., \cite{IbarraK75} for the knapsack problem.

\paragraph{Dynamic programming for integral profits.}

Let $W(i,q,r)$ denote the minimum weight of objects 
among $\{ x_1 , \ldots , x_i \}$ that can attain profit $r$ using a weight-improvement budget of at most $q$. The following observations are immediate. 
\begin{enumerate}
\item $W(i, 0, r)$ is the standard version of the knapsack
problem where the weights are $w_i^0$ and $r \leq P := \sum_i p_i$.
The optimal objective value is given by $\max \{ r\in\mathbb{N}\ :\ W(n,0,r)  \leq B \}$
\item $W(i,q+1,r) \leq W(i,q,r)$, i.e., more weight reductions cannot decrease the the value of the solution.
\end{enumerate}

We can now write the following recurrence for $1 \leq n, 1 \leq q \leq C, 
1 \leq r \leq P$:
%
For an item $i$, there are weight reductions with increasing 
costs $c_i^1 \leq \ldots \le c_i^{j(i)}$ that yields (decreasing) weights
$w_i^1 \geq \ldots \geq w_i^{j(i)}$.  We can now write the following dynamic
programming recurrence
\begin{equation}
W(i,q,r) = \min
\begin{cases}
W(i-1 , q, r), & \text{\small(do not use $i$)}\\[2mm]
W(i-1, q, r - p_i) + w_i^0, & \text{\small($i$ is not reduced)}\\[2mm]
 W(i-1, q- c_i^1 , r- p_i ) + w_i^1,  &  \text{\small(it costs $c_i^1$ for $w_i^1$)} \\[2mm]
W(i-1, q- c_i^2 , r- p_i ) + w_i^2,  &  \text{\small(it costs $c_i^2$ for $w_i^2$)} \\[2mm]
\ldots, &  \text{\small(it costs $c_i^\ell$ for $w_i^\ell$)} \\[2mm]
W(i-1, q- c_i^{j(i)} , r- p_i ) + w_i^{j(i)} & \text{\small(it costs $c_i^{j(i)}$ for $w_i^{j(i)}$)} \hspace{-3mm}\\
\end{cases}
\label{dynprog2}
\end{equation}
It may be noted that reducing the weight of the $i$-th item and not choosing it is worse
than the first term, and hence need not be considered.
Let $W(i, q, r) = -\infty$ for $q < 0$ so that we do not consider terms in
the dynamic programming where the improvement cost exceeds the current improvement budget.
Use the base case as $W(1,0,r) = w_1^0$ for $r = p_1$ and 0 otherwise. 
We assume that $c_i^j$ for all $i,j$ are integral and 
each entry of the table can be computed in $Q:=\max_i j(i)$ steps. The resulting dynamic programming algorithm is presented as Algorithm~\ref{algo-3}.

\begin{algorithm}[ht]
\caption{Pseudo-polynomial Algorithm for (iK-dm)}\label{algo-3}
\begin{algorithmic}[1]
\Require{ A problem instance of (iK-dm) with integer weights.}
\State  Initialize the table $W = n \times C \times P$ to $- \infty$. Set $W(1,0,r) = w_1^0$ for $r = p_1$, and 0 otherwise.
\For{ $q = 0$ to $C$ }
\For{ $i = 1$ to $n$ }
\For{ $r = 1$ to $P$ }
\State Set $W(i,q,r)$ according to Equation~\eqref{dynprog2}.
\EndFor
\EndFor
\EndFor
\State \Return  $\arg\max_r \{ W(n,C,r)  \leq B \}$.
\end{algorithmic}
\end{algorithm}

\begin{lemma}
Algorithm~\ref{algo-3} takes time $\mathcal{O}(n C Q P)$.
\end{lemma}
\begin{proof}
Each entry can be computed
in $Q$ steps where the order of computation proceeds from $q = 0$ to $C$
and for a fixed $q$, we compute the entries in increasing order of $i$ 
and $r$ (for a fixed $i$, in increasing order of $r$). 
\end{proof}

\paragraph{Faster approximation algorithms using profit scaling.}

Using profit scaling, we now convert the previous algorithm into a more 
efficient version by compromising with an approximation factor in the
objective function.   
Suppose we want to compute a solution with an objective value of at least $(1 - \epsilon )\text{iK-dm}^*$.
We use the scaling method, namely for any object $x_i$, we consider its new profit
$p'_i  = \lfloor \frac{ p_i }{K} \rfloor$ 
where $K = \frac{\epsilon \cdot p_{\max}}{n}$ and use this to run the dynamic programming equation. Note that any
$K \leq \epsilon \cdot \text{iK-dm}^* /n$ suffices for this purpose.


\begin{algorithm}[ht]
\caption{PTAS for (iK-dm)}\label{algo-fptas}
\begin{algorithmic}[1]
\Require{ A problem instance of (iK-dm), and $\epsilon > 0$.}
\State Set $K=\frac{\epsilon p_{\max}}{n}$.
Let $p'_i  = \lfloor \frac{ p_i }{K} \rfloor$
\State  Solve the instance $\text{iK-dm}(p')$ using Algorithm~\ref{algo-3}. Let $(x,y)$ be the resulting solution.
\State \Return $(x,y)$
\end{algorithmic}
\end{algorithm}

Using $p'_{\max} = \mathcal{O}( n/\epsilon )$, 
the running time of the resulting Algorithm~\ref{algo-fptas} is $\mathcal{O}(nC Q\cdot n \cdot \frac{n}{\epsilon})$.
which is similar to the classic FPTAS for Knapsack \cite{Vazirani:2001}.

%
%
%
\begin{theorem}
The dynamic programming algorithm for the knapsack problem with multiple, discrete weight improvements 
returns a solution with objective value at least $(1 -\epsilon )\text{iK-dm}^*$ in $\mathcal{O}(\frac{n^3\cdot Q C}{\epsilon})$ time.
\end{theorem}

\begin{remark}
If the total improvement budget $C$ is bounded by a polynomial in $n$, this is even an FPTAS. For general $C$, no FPTAS exists, as can be easily seen by a reduction from the 2-partition problem.
\end{remark}

\paragraph{A polynomial-time 3-approximation algorithm.}
We now present a polynomial time approximation algorithm for (iK-dm).
This is achieved at a cost of relaxing the approximation to factor $3$.

To this end, we consider the LP relaxation obtained by relaxing constraints (\ref{gcr-int1}) and (\ref{gcr-int2}) to
\begin{align}
& x_i  \le  1 & \forall i\in\{1,\ldots,n\} \label{gcr-b2}\\
& y_i^{j(i)} \ge 0 & \forall i\in\{1,\ldots,n\} \label{gcr-c}
\end{align}
Thus, there are $j(i)+2$ constraints associated with every item -- obtained from constraints (\ref{gcr-b2}),  (\ref{gcr-c}) above combined with
constraints (\ref{gcr-a})  recalled below: 
\begin{align*}
&y_i^1 \le x_i & \forall i\in\{1,\ldots,n\} \\
& y_i^{\ell+1} \le y_i^{\ell} & \forall i\{1,\ldots,n\},\ \ell\in\{1,\ldots,j(i)-1 \}
\end{align*}
In addition, we have the knapsack constraints w.r.t. $B$ and $C$. Therefore, the total number of constraints is $$2 + \sum_{i=1}^n \left( j(i)+2 \right).$$
As there are $j(i)+1$ variables associated with every item, the total number of variables is
$$\sum_{i=1}^{n} (j(i)+1).$$
Moreover the LP is bounded. 
Therefore the number of tight constraints in an optimal basic feasible solution must be
$$\sum_{i=1}^{n} (j(i)+1).$$
This implies that at most $n+2$ constraints can be non-tight in a basic feasible solution.
Let us see how the items contribute non-tight constraints.
The important observation is that for any item $i$, all $j(i)+2$ constraints cannot be simultaneously tight as this would imply that
$$0 = y_i^{j(i)} = \ldots = y_i^{\ell+1} = y_i^{\ell} = \ldots = x_i=1$$ which is not possible.
Thus every item must contribute at least one non-tight constraint. Since the total number of non-tight constraints can be at most $n+2$,
at most two items can contribute more than one non-tight constraint; all the remaining items must contribute only
one non-tight constraint.

Now consider an item that contributes exactly one non-tight constraints. Then one of the cases holds
depending on which constraint is non-tight:
\begin{itemize}
\item{} If $y_i^{j(i)} > 0$, then $$y_i^{j(i)} = \ldots = y_i^{1} = x_i=1.$$
\item{} If $y_i^{\ell+1} < y_i^\ell$ for some $1 \le \ell \le j(i)-1$, then $$0 = y_i^{j(i)} = \ldots = y_i^{\ell+1} \mbox{ and } y_i^\ell = \ldots = y_i^1 = x_i=1.$$
\item{} If $y_i^1 < x_i$, then $$0 = y_i^{j(i)} = \ldots = y_i^1 \mbox{ and } x_i = 1.$$
\item{} If $x_i < 1$, then $$0 = y_i^{j(i)} = \ldots = y_i^1 = x_i.$$
\end{itemize}
Thus, if an item contributes exactly one non-tight constraint, then all the variables
associated with this item must be integral. We call such items to be {\em integral}.

Now, since at most two items can contribute more than one non-tight constraint, it implies that there can be at most two items that are not integral.
We create three integral solutions from the LP solution: One consisting of all the integral items in the LP solution and one each corresponding to the
two items that are not integral. Clearly the one with the best profit is a $3$-approximate solution. We summarize this approach as Algorithm~\ref{algo-polynew}.

\begin{algorithm}[ht]
\caption{}\label{algo-polynew}
\begin{algorithmic}[1]
\Require{ A problem instance of (iK-dm).}
\State Compute an optimal basic solution of the LP relaxation of (iK-dm). Let $J$ be the indices of integral items that are packed, and let $Z$ denote the accompanying vector of integral improvements. Let $x_{f_1}$ and $x_{f_2}$ denote the fractional items of the solution, if they exist.
\State \Return $\arg\max\{ p(J,Z), p(x_{f_1}), p(x_{f_2}) \}$
\end{algorithmic}
\end{algorithm}

\begin{theorem}
There is a 3-approximation algorithm for (iK-dm) that runs in polynomial time.
\end{theorem}

\section{Improved Results for Unit Improvement Costs}

\label{sec:spec}

\subsection{A Linear-time 3-Approximation Algorithm}

We consider the single-level weight improvement case with all improvements costs equal to 1.
We develop an approach that is based on creating a
cardinality-constrained knapsack (CKP) problem. 
In particular, given an instance of (iK-ds) with unimproved weights $w$ and improved weights $\what$, we create a CKP
instance as in the formulation of (iK-ds')
by doubling all items; i.e., we create an instance consisting
of $2n$ items, where the first $n$ items have weight $w$, and the
next $n$ items have weight $\what$. As a slight modification of the
original CKP definition, we assume that the cardinality constraint
only applies to the items with weight $\what$. The problem we consider is 
denoted as
\begin{align*}
\text{(CKP')} \hspace{1.5cm} \max\ &\sum_{i=1}^n p_i( x_i + \hat{x}_i) \\
\text{s.t.}\ &\sum_{i=1}^n w_i x_i + \sum_{i=1}^n \hat{w}_i \hat{x}_i \le B\\ 
&\sum_{i=1}^n \hat{x}_i \le k\\ 
&x_i \in \{0,1\} & \forall i\in\{1,\ldots,n\} \\
&\hat{x}_i \in\{0,1\} & \forall i\in\{1,\ldots,n\}
\end{align*}
CKP' is a relaxation of (iK-ds'), as the coupling constraints~\eqref{linear1-lp4} are ignored. Hence $\text{CKP'}^* \ge \text{iK-ds}^*$. 
Solving the LP-relaxation of CKP' 
results 
in a basic solution with a set of integer variables $J^*_I = J_I \cup \hat{J}_I$ 
and a set of fractional variables $J^*_F$. Note that, as before, $|J^*_F| \le 2$. 

\begin{lemma}
Let $(x,\hat{x})$ be a basic solution of the LP relaxation of CKP'. 
If there are two fractional variables, then these are $\hat{x}_i$ and $\hat{x}_j$ with $\hat{x}_i + \hat{x}_j = 1$ for some $i,j$.
\end{lemma}
\begin{proof}
Let there be two fractional variables. We consider the following cases:
\begin{enumerate}
\item[(1.)] If $x_i$ and $x_j$ are fractional, we can improve the solution by increasing the variable with better profit to weight ratio, and decreasing the other, until one of them is either $0$ or $1$.
\item[(2.)] If $x_i$ and $\hat{x}_j$ are fractional, the cardinality constraint cannot be tight.
We hence can improve the solution by increasing the variable with better profit to weight ratio 
as in $1$. until either one of the variables reaches $0$ or $1$.
\item[(3.)] If $\hat{x}_i$ and $\hat{x}_j$ are fractional, and the cardinality constraint is not tight, we may proceed as in (2.), until one of the variables reaches $0$ or $1$, or the cardinality constraint becomes tight.
\item[(4.)] If $\hat{x}_i$ and $\hat{x}_j$ are fractional, and the cardinality constraint is tight, we have  $\hat{x}_i + \hat{x}_j = 1$.
\end{enumerate}
\end{proof}

We use these properties to construct the following feasible solutions for (iK-ds):
\begin{enumerate}
\item If $J^*_F = \emptyset$, we construct the two solutions $(J_I,\emptyset)$ and $(\emptyset,\hat{J}_I )$.
\item If $J^*_F = \{ i \}$, we use the three solutions $(\emptyset, \{i\})$, $(J_I,\emptyset)$, and $(\emptyset,\hat{J}_I )$.
\item Finally, if $J^*_F = \{i,j\}$, where w.l.o.g. $\what_i \ge \what_j$, we use $(\emptyset, \{i\})$, $(J_I,\emptyset)$, and $(\emptyset,\hat{J}_I\cup \{j\} )$.
\end{enumerate}

Note that these solutions are feasible for (iK-ds), and the sum of their objective values is 
larger than $\text{iK-ds}^*$. Thus, choosing the solution with the maximal objective value yields a 3-approximation. 
We recapitulate this approach in Algorithm~\ref{algo-new}.

\begin{algorithm}[ht]
\caption{}\label{algo-new}
\begin{algorithmic}[1]
\Require{ A problem instance of (iK-ds) with unit improvement costs.}
\State Solve the LP relaxation of CKP'.
Let $J^*_I=J_I \cup \hat{J}_I$ and $J^*_F$ 
denote the item indices with integer values packed with original or reduced weights, 
and the item indices with fractional values in a basic solution.
\If{$J^*_F=\emptyset$}
  \State \Return $\arg\max\{ p(J_I,\emptyset), p(\emptyset,\hat{J}_I) \}$.
\ElsIf{$J^*_F=\{i\}$ for some $i\in\{1,\ldots,n\}$}
\State \Return $\arg\max\{ p(\emptyset, \{i\}), p(J_I,\emptyset), p(\emptyset,\hat{J}_I ) \}$
\ElsIf{$J^*_F=\{i,j\}$ for some $i,j\in\{1,\ldots,n\}$ with $\what_i \ge \what_j$}
\State \Return $\arg\max\{ p(\emptyset, \{i\}), p(J_I,\emptyset), p(\emptyset,\hat{J}_I\cup \{j\} ) \}$
\EndIf
\end{algorithmic}
\end{algorithm}

Note that the LP relaxation of CKP' can be solved in linear time \cite{Megiddo93lineartime}. Thus we can state the following theorem.

\begin{theorem}
Algorithm~\ref{algo-new} has an approximation ratio of at most 3 for (iK-ds) with unit improvement costs, and runs in linear time.
\end{theorem}

\subsection{A Polynomial-time 2-Approximation Algorithm}

We now show that a factor $2$ approximation for the unit improvement case can be achieved by running in polynomial time.
Recall that for the generalized case, we are able to achieve a factor $3$-approximation algorithm by 
considering the LP relaxation of the problem and characterizing the basic feasible solutions of the relaxed LP.
We show that for the special case of one improvement per item with unit costs, we can better characterize the basic feasible solutions 
of the relaxed LP yielding an improved factor $2$ approximation.
For this, we consider the the LP relaxation of (\ref{linear2-lp1}--\ref{linear2-lp5}). 
Note that the linear-time result of \cite{Megiddo93lineartime} does not apply here due to the 
non-constant 
number of constraints. The LP relaxation can be written as:
\begin{align}
\label{linear-lp2}
\max\ &\sum_{i=1}^n c_ix_i\\
\label{k-cr}
\text{s.t.}\ &\sum_{i=1}^n w_i x_i \le B + \sum_{i=1}^n \overline{w}_i y_i\\
\label{q-cr}
&\sum_{i=1}^n y_i \le k\\
\label{cr-a}
&y_i \le x_i & \forall i\in \{1,\ldots,n\}\\
\label{cr-b}
&x_i \le 1 & \forall i\in\{1,\ldots,n\}\\
\label{cr-c}
&y_i \ge 0 & \forall i\in\{1,\ldots,n\}
\end{align}
The LP has $2n$ variables and $3n+2$ constraints comprising of the knapsack-constraint (\ref{k-cr}), 
the $k$-constraint (\ref{q-cr}) and three constraints for each item, (\ref{cr-a}), (\ref{cr-b}) and (\ref{cr-c}). 
Observe that the item constraints imply that the feasible region is bounded.
For any 
basic feasible solution there must be $2n$ linearly independent constraints that are tight.
We categorize the items based on the number of tight constraints among (\ref{cr-a}),(\ref{cr-b}),
and (\ref{cr-c}) it can contribute, see Table~\ref{tab-1}.

\begin{table*}
\centering
\begin{tabular}{|c|c|c|c|c|}
\hline
\ Case \ & \ Type \ & \ Num of Tight \ & \ Tight Constraints \ & \ Num of  non-integral  \ \\
 & & \ Constraints \ &  & \ variables \ \\
\hline
i & T1 & 0 & None & 2 \\
ii & T2 & 1 & $x_i=y_i$ & 2 \\
iii & T3 & 1 & $y_i=1$ & 1 \\
iv & T3 & 1 & $x_i=1$ & 1 \\
v & T4 & 2 & $x_i=1, y_i=0$ & 0 \\
vi & T4 & 2 & $y_i=0, x_i=y_i$ & 0 \\
vii & T4 & 2 & $x_i=1, x_i=y_i$ & 0 \\
viii & T5 & 3 & \ \ $y_i=0, x_i=y_i, x_i=1$ \ \ & Not Possible \\
\hline
\end{tabular}
\caption{Item categorization.}\label{tab-1}
\end{table*}

We observe that an item cannot contribute more than two tight constraints, i.e., constraints
(\ref{cr-a}), (\ref{cr-b}) and (\ref{cr-c}) cannot simultaneously be all tight for the same 
item (case viii). 

We consider two scenarios: either the $k$-constraint (\ref{q-cr}) is tight or not.

In case it is not tight, then discounting the knapsack constraint, we see that  
$2n-1$ of the tight constraints must be constraints of type
(\ref{cr-a}), (\ref{cr-b}) and (\ref{cr-c}). 
This implies that at least $n-1$ items must be of type T4. 
Therefore $n$ items can contribute at least $2n-1$ tight constraints only under one of the following scenarios:
\begin{itemize}
\item[A.] $n$ items of type T4
\item[B.] $(n-1)$ items of type T4 and $1$ item of type T1, T2 or T3.
\end{itemize}

In case, the $k$-constraint is tight, then
discounting the $k$-constraint and the knapsack constraint, we see that $2n-2$ of the tight constraints 
must come from constraints of type
(\ref{cr-a}), (\ref{cr-b}) and (\ref{cr-c}). 
This implies that at least $n-2$ items must be of type T4. 
Therefore, $n$ items can contribute at least $2n-2$ tight constraints only under one of the following scenarios:
\begin{itemize}
\item[C.] $n$ items of type T4
\item[D.] $(n-1)$ items of type T4 and $1$ item of type T1, T2 or T3.
\item[E.] $(n-2)$ items of type T4 and $2$ items of type T2 or T3
\end{itemize}

In Cases A and C, all the variables are integral and therefore the solution is 
integral yielding the exact optimal.

In Cases B and D, we form 2 solutions -- one consisting of all the type T4 items (which are already integral) and
the other consisting of the remaining item that is either of type T1, T2 or T3 in the weight-reduced form. 
The first solution is clearly integral feasible, as it is a subset of the fractional optimal. The second solution is integral as
every item under consideration is feasible in its weight-reduced form.
We simply pick the better of the two solutions yielding a 2-approximate solution.

In case E, let $i$ and $j$ be the two items of type T1/T2/T3.
We note that the $k$-constraint must be tight. Thus, we have that $y_i + y_j=1$.
Without loss of generality, let $\overline{w}_i \le \overline{w}_j$.
We therefore form two solutions -- one consisting of all the type T4 items along with $i$ in weight-reduced form and the other consisting of
$j$ in weight-reduced form. We again pick the best of the two
solutions to yield a 2-approximation.

Thus we obtain a 2-approximation algorithm.
Note that unlike the $3$-approximation algorithm for the generalized case, the relaxation to unit costs allows us to utilize the tightness of the $k$-constraint
in a meaningful way to obtain a better approximation.

\begin{algorithm}[ht]
\caption{}\label{algo-2}
\begin{algorithmic}[1]
\Require{ A problem instance of (iK-ds) with unit improvement costs.}
\State Compute an optimal basic solution of the LP relaxation of (iK-ds). Let $(J^w_i,J^{\overline{w}}_i)$, $i=1,2,3,4$, denote the unimproved and improved item indices of type $T_i$, respectively.
\If{$|T_4| = n$}
\State \Return the (optimal) iK-ds solution $(J^w_4,J^{\overline{w}}_4)$.
\ElsIf{$|T_4| = n-1$ and $|T_1\cup T_2 \cup T_3| = \{i\}$}
\State \Return $\arg\max \{ p(J^w_4,J^{\overline{w}}_4),\ p(\emptyset, \{i\})\}$.
\ElsIf{$|T_4| = n-2$ and $|T_2 \cup T_3| = \{i,j\}$}
\State W.l.o.g., let $\overline{w}_i \le \overline{w}_j$.
\State \Return $\arg\max \{ p(J^w_4,J^{\overline{w}}_4\cup\{i\}),\ p(\emptyset, \{j\})\}$.
\EndIf
\end{algorithmic}
\end{algorithm}

\begin{theorem}
Algorithm~\ref{algo-2} has an approximation ratio of at most 2 for (iK-ds) with unit improvement costs, and runs in polynomial time; more specifically, in time required to solve an LP.
\end{theorem}

\section{Extensions and Conclusion}
\label{sec:conclusion}

We introduced the improvable knapsack problem, where a separate budget is available to improve the weights of items. While network improvement problems have been thoroughly studied, this is the first such approach to knapsacks.

The previous results can also be applied to more general improvable knapsack problems, i.e., when also profit improvements are included. We briefly review these cases in the following.
\begin{itemize}
\item The single-level continuous profit improvement case can be modeled with the following mixed-integer program:
\begin{align*}
\max\ &\sum_{i=1}^n p_i x_i + \sum_{i=1}^n \overline{p}_i z_i \\
\text{s.t. } & \sum_{i=1}^n w_i x_i \le B \\
&\sum_{i=1}^n d_i z_i \le D \\
& z_i \le x_i & \forall i\in\{1,\ldots,n\} \\
& z_i \ge 0, x_i \in\{0,1\} & \forall i\in\{1,\ldots,n\}
\end{align*} 
Here, variables $z$ are used to model profit improvement of items.
Note that the structure of this problem is very similar to the single-level continuous weight improvement case: As before, an optimal choice for the improvements $z$ can be found by sorting the items by $d_i/\overline{p}_i$ if the variables $x$ are fixed. Using the same arguments as in Section~\ref{sec:contweights}, there exist a critical item index $k$ also for profit improvements. If $k$ is the critical item, $x$ a feasible solution, then we can find its profit by calculating
\begin{align*}
\sum_{i=1}^n p_i x_i + \sum_{i=1}^n \overline{p}_i z_i &= \sum_{i=1}^n p_i x_i + \sum_{i=1}^{k-1} \overline{p}_i z_i + \overline{p}_k z_k \\
&= \sum_{i=1}^{k-1} (p_i + \overline{p}_i) x_i + \sum_{i=k}^n p_i x_i + \overline{p}_k z_k \\
&= \sum_{i=1}^{k-1} (p_i + \overline{p}_i) x_i + \sum_{i=k}^n p_i x_i + \overline{p}_k \frac{D-\sum_{i=1}^{k-1} d_i x_i}{d_k} \\
&= \sum_{i=1}^{k-1} (p_i + \overline{p}_i - \overline{p}_k\frac{d_i}{d_k}) x_i + p_i + \sum_{i=k+1}^n p_i x_i + D\frac{\overline{p}_k}{d_k}
\end{align*}
Thus, solving $n$ problems similar to $P(k)$ also suffices to find an optimal solution, which gives us a PTAS for this problem.

\item For combined models of the form
\begin{align*}
\max\ &\sum_{i=1}^n p_i x_i + \sum_{i=1}^n \overline{p}_i z_i \\
\text{s.t. } & \sum_{i=1}^n w_i x_i - \sum_{i=1}^n \overline{w}_i y_i \le B \\
&\sum_{i=1}^n c_i y_i \le C \\
&\sum_{i=1}^n d_i z_i \le D \\
& y_i \le x_i & \forall i\in\{1,\ldots,n\} \\
& z_i \le x_i & \forall i\in\{1,\ldots,n\} \\
& y_i, z_i \ge 0, x_i \in\{0,1\} & \forall i\in\{1,\ldots,n\}
\end{align*} 
it is possible again to find optimal values for $y$ and $z$ for fixed variables $x$ by sorting the items by $c_i/\overline{w}_i$ and $d_i/\overline{p}_i$, respectively. This results in critical items for both profit and weight. In this case, a solution approach might consider all $n^2$ possible index combinations for profit and weight. In future research, this approach needs to be considered in detail.

\item For discrete improvements in both the profits and the weights the dynamic program from Section~\ref{sec:dp} can be immediately extended, leading to a PTAS for the general case.
\end{itemize}
More further research includes the analysis of improvable knapsack problems with a combined improvement budget for profit and weight improvement, as well as the extension to related combinatorial optimization problems, such as shortest paths.

Finally, improvable problems also play a role when computing the {\it query competitiveness} of an algorithm for an uncertain problem that allows queries to improve the problem knowledge (see \cite{sea}). It remains open how an algorithm for the improvable problem may be extended to a competitive algorithm for the uncertain problem.

\bibliography{references}
\bibliographystyle{alpha}

\end{document}